\documentclass[reqno,11pt]{amsart}
\usepackage{latexsym}
\usepackage{amsfonts}
\usepackage{amssymb}
\usepackage{mathrsfs}
\usepackage[all]{xy}
\usepackage{bbm}

\newtheorem{theorem}{Theorem}[section]
\newtheorem{corollary}[theorem]{Corollary}
\newtheorem{lemma}[theorem]{Lemma}
\newtheorem{proposition}[theorem]{Proposition}
\theoremstyle{definition}
\newtheorem{definition}[theorem]{Definition}
\numberwithin{equation}{section}

\newcommand{\RR}{\mathbb{R}}
\newcommand{\NN}{\mathbb{N}}
\newcommand{\QQ}{\mathbb{Q}}
\newcommand{\ZZ}{\mathbb{Z}}
\newcommand{\mt}{\mapsto}

\newcommand{\cT}{\mathcal{T}}
\newcommand{\cS}{\mathcal{S}}
\newcommand{\cR}{\mathcal{R}}
\newcommand{\cF}{\mathcal{F}}
\newcommand{\cP}{\mathscr{P}}
\newcommand{\cV}{\mathcal{V}}
\newcommand{\cA}{\mathbf{A}}
\newcommand{\cB}{\mathbf{B}}
\newcommand{\cC}{\mathbf{C}}
\newcommand{\cU}{\mathcal{U}}
\newcommand{\cL}{\mathcal{L}}

\newcommand{\sF}{\mathscr{F}}

\newcommand{\Bad}{\mathbf{Bad}}

\newcommand{\bx}{\mathbf{x}}
\newcommand{\suc}{\mathrm{suc}}
\newcommand{\diag}{\mathrm{diag}}
\newcommand{\SL}{\mathrm{SL}}

\begin{document}

\title[Two-dimensional badly approximable vectors and Schmidt's game]
{Two-dimensional badly approximable vectors \\ and Schmidt's game}

\author{Jinpeng An}
\address{LMAM, School of Mathematical Sciences, Peking University, Beijing, 100871, China}
\email{anjinpeng@gmail.com}

\thanks{Research supported by NSFC grant 10901005/11322101 and FANEDD grant 200915.}

\begin{abstract}
We prove that for any pair $(s,t)$ of nonnegative numbers with $s+t=1$, the set of two-dimensional $(s,t)$-badly approximable vectors is winning for Schmidt's game. As a consequence, we give a direct proof of Schmidt's conjecture using his game.
\end{abstract}

\maketitle

\section{Introduction}

\subsection{Schmidt's conjecture and Schmidt's game}
Given a pair $(s,t)$ of nonnegative numbers with $s+t=1$, a two-dimensional vector $(x,y)\in\RR^2$ is said to be \emph{$(s,t)$-badly approximable} if $$\inf_{q\in\NN}\max\{q^s\|qx\|,q^t\|qy\|\}>0,$$ where $\|\cdot\|$ denotes the distance of a number to the nearest integer. As a natural generalization of badly approximable numbers, the set of $(s,t)$-badly approximable vectors, denoted by $\Bad(s,t)$, is a fundamental object of study in simultaneous Diophantine approximation. It is well-known that $\Bad(s,t)$ has Lebesgue measure zero and full Hausdorff dimension in $\RR^2$ (see \cite{PV}). In the early 1980's, W. M. Schmidt \cite{Sc3} conjectured that $\Bad(\frac{1}{3},\frac{2}{3})\cap\Bad(\frac{2}{3},\frac{1}{3})\ne\emptyset$. Schmidt's conjecture was recently proved by D. Badziahin, A. Pollington and S. Velani \cite{BPV}. In fact, they proved a much stronger theorem, which states that certain countable intersection (in particular, any finite intersection) of $\Bad(s_n,t_n)$ has full Hausdorff dimension.

On the other hand, in the 1960's, Schmidt \cite{Sc1} introduced a game played on a complete metric space by two players. Winning sets for Schmidt's game has very nice properties. For example, a winning subset of an Euclidean space has full Hausdorff dimension. More importantly, a countable intersection of $\alpha$-winning sets is still $\alpha$-winning. Schmidt \cite{Sc1,Sc2} showed that $\Bad(\frac{1}{2},\frac{1}{2})$ is $1/2$-winning. As such, it is natural to expect that $\Bad(s,t)$ is a winning set in general, and thus Schmidt's conjecture can be proved directly using his game. This expectation was raised explicitly by Kleinbock \cite{Kl} (see also \cite{KW1,Mo}). For similar questions and results for higher-dimensional vectors and matrices, see, for example, \cite{Kl,KW0,KW2,PV,Sc1.5}.

\subsection{Proving Schmidt's conjecture using his game}
The goal of proving Schmidt's conjecture using his game was partly achieved in \cite{An}. It was proved there that if $x\in\RR$ is badly approximable, then the set of $y\in\RR$ such that $(x,y)$ is $(s,t)$-badly approximable is a winning subset of $\RR$. As a consequence, any countable intersection of $\Bad(s_n,t_n)$ has full Hausdorff dimension. In this paper, we prove that $\Bad(s,t)$ itself is a winning subset of $\RR^2$, thus give a more direct proof of Schmidt's conjecture. Our main theorem is as follows.

\begin{theorem}\label{T:main}
For any $s,t\ge0$ with $s+t=1$, the set $\Bad(s,t)$ is $(24\sqrt{2})^{-1}$-winning.
\end{theorem}

Theorem \ref{T:main} implies stronger full dimension results. For example, since a countable intersection of images of $\alpha$-winning sets under uniformly bi-Lipschitz homeomorphisms is still winning (see \cite{Da2,Sc1}), we obtain the following result.

\begin{corollary}
Let $(s_n,t_n)_{n=1}^\infty$ be a sequence of pairs of nonnegative numbers with $s_n+t_n=1$, and let $(f_n)_{n=1}^\infty$ be a sequence of uniformly bi-Lipschitz homeomorphisms of $\RR^2$, that is, there exists $M\ge1$ such that
$$M^{-1}|\bx_1-\bx_2|\le |f_n(\bx_1)-f_n(\bx_2)|\le M|\bx_1-\bx_2|, \qquad \forall \bx_1,\bx_2\in\RR^2, n\ge1,$$ where $|\cdot|$ is the Euclidean norm.
Then the set $\bigcap_{n=1}^\infty f_n(\Bad(s_n,t_n))$ has full Hausdorff dimension in $\RR^2$.
\end{corollary}

It should be noted that several stronger variants of Schmidt's game have been defined and used to problems in Diophantine approximation (see, for example, \cite{Mc,BFKRW}). By using the main lemma in a previous version of this paper (a weaker form of Corollary \ref{C:strip} below), it has been proved in \cite{NS} that $\Bad(s,t)$ is hyperplane absolute winning in the sense of \cite{BFKRW}.

\subsection{Relationship to homogeneous dynamics}
As is well known, badly approximable vectors correspond to certain bounded trajectories on the homogeneous space $\SL_3(\RR)/\SL_3(\ZZ)$. For $(x,y)\in\RR^2$, we denote $h_{(x,y)}=\begin{pmatrix}1&0&x\\0&1&y\\0&0&1\end{pmatrix}$. Let $H\cong\RR^2$ be the subgroup of $G=\SL_3(\RR)$ consisting of matrices of the form $h_{(x,y)}$. By a \emph{ray} in $G$, we mean a set of the form $F^+=\{g_u:u\ge0\}$, where $u\mt g_u$ is a one-parameter subgroup of $G$. Consider rays of the form
\begin{equation}\label{E:ray}
F_{(s,t)}^+=\{\diag(e^{su},e^{tu},e^{-u}):u\ge0\}, \qquad s,t\ge0, s+t=1.
\end{equation}
Then $(x,y)$ is $(s,t)$-badly approximable if and only if the trajectory $F_{(s,t)}^+h_{(x,y)}\Gamma$ is bounded in $G/\Gamma$, where $\Gamma=\SL_3(\ZZ)$ (see \cite{Da1,Kl}). Let $D$ be the group of diagonal matrices in $G$, and consider its sub-semigroup
$$D^+=\{\diag(e^{u_1},e^{u_2},e^{-u_1-u_2}):u_1,u_2\ge0\}.$$
Then any ray in $D^+$ is of the form \eqref{E:ray}. Thus Theorem \ref{T:main} amounts to the statement that for any ray $F^+$ in $D^+$, the set of $h\in H$ for which $F^+h\Gamma$ is bounded is $(24\sqrt{2})^{-1}$-winning.
In a much more general context, the winning property for sets of this form with respect to a modified Schmidt game was established in \cite{KW2}.

It was proved in \cite{EKL} that the set of $h\in H$ for which $D^+h\Gamma$ is bounded has Hausdorff dimension zero (note that $D^+h_{(x,y)}\Gamma$ is bounded if and only if $(x,y)$ violates Littlewood's conjecture $\inf_{q\in\NN}q\|qx\|\|qy\|=0$). A conjecture from \cite{Go} states that for any two rays $F^+_1,F^+_2$ in $D$, there exists $g\in G$ such that $F^+_1g\Gamma$ and $F^+_2g\Gamma$ are bounded but $Dg\Gamma$ is unbounded. D. Kleinbock observed that if $F^+_1$ and $F^+_2$ lie in opposite Weyl chambers, the arguments in \cite{KM} can be adapted to prove that the set of $g\in G$ satisfying the conjecture has full Hausdorff dimension. On the other hand, the main theorem in \cite{BPV} implies that if $(F_n^+)_{n=1}^\infty$ is a sequence of rays in $D^+$ satisfying a  certain technical assumption, then the set
\begin{equation}\label{E:Go}
\{h\in H:F_n^+h\Gamma \text{ is bounded, } \forall n\ge1\}
\end{equation}
has full Hausdorff dimension in $H$. It follows from Theorem \ref{T:main} that without the technical assumption, the set \eqref{E:Go} is winning.

\subsection{On the proof of Theorem \ref{T:main}}
Unlike previous proofs of the winning property (except for \cite{An}), our proof of Theorem \ref{T:main} is non-constructive. In other words, it does not give an explicit winning strategy, but only shows the existence of a winning strategy. This is reflected in the proof of Proposition \ref{P:tree} below, where we use K\"{o}nig's lemma in graph theory to show the existence of a certain subtree that corresponds to a winning strategy.

A crucial ingredient in establishing Theorem \ref{T:main} is the height function on rational points given by \eqref{E:height}. It relies not only on the rational point itself, but also on a rational line passing through the point which is ``small" in the sense of Lemma \ref{L:aug} below. The height function is used to define a partition of rational points, which in turn gives rise to a Cantor-like set contained in $\Bad(s,t)$. We prove in Corollary \ref{C:strip} that, roughly speaking, in the construction of the Cantor-like set, at the $n$-th step we need only to remove squares that intersect small neighborhoods of $n$ lines. This implies that the Cantor-like set is ``fat" enough so that it is winning for Schmidt's game.

In order to simplify the presentation and resort to K\"{o}nig's lemma directly, it is convenient to represent squares used in the construction of the Cantor-like set as vertices of a rooted tree, and color the vertices in a regular manner. In Section \ref{S:2}, we provide preliminaries on colored rooted trees. Theorem \ref{T:main} is proved in Sections 3 and 4.

\section{Regular colorings of rooted trees}\label{S:2}

We use the same notation and conventions for rooted trees as in \cite{An}. For example, we identify a rooted tree $\cT$ with the set of its vertices, and denote the set of vertices of level $n$ by $\cT_n$. For $\tau\in\cT$, let $\cT(\tau)$
denote the rooted tree formed by the descendants of $\tau$, and $\cT_{\suc}(\tau)$ denote the set of successors of $\tau$. For $\cV\subset\cT$, denote $\cT_{\suc}(\cV)=\bigcup_{\tau\in\cV}\cT_{\suc}(\tau)$. By convention, a subtree has the same root as the ambient tree.

Let $D\in\NN$. A \emph{$D$-coloring} of a rooted tree $\cT$ is a map $\gamma:\cT\to\{1,\ldots,D\}$. For $\cV\subset\cT$ and $1\le i\le D$, we denote $\cV^{(i)}=\cV\cap\gamma^{-1}(i)$. Let $N\in\NN$ be an integer multiple of $D$, and suppose that $\cT$ is $N$-regular, that is, $\#\cT_{\suc}(\tau)=N$ for every $\tau\in\cT$. We say that a $D$-coloring of $\cT$ is \emph{regular} if for any $\tau\in\cT$ and $1\le i\le D$, we have $\#\cT_{\suc}(\tau)^{(i)}=N/D$. The following two types of subtrees are of interest to us.

\begin{definition}
Let $\cT$ be an $N$-regular rooted tree with a regular $D$-coloring, and let $\cS\subset\cT$ be a subtree.
\begin{itemize}
  \item The subtree $\cS$ is of \emph{type (I)} if for any $\tau\in\cS$ and $1\le i\le D$, we have $\#\cS_{\suc}(\tau)^{(i)}=1$.
  \item The subtree $\cS$ is of \emph{type (II)} if for any $\tau\in\cS$, there exists $1\le i(\tau)\le D$ such that $\cS_{\suc}(\tau)=\cT_{\suc}(\tau)^{(i(\tau))}$.
\end{itemize}
\end{definition}

Roughly speaking, in the proof of Theorem \ref{T:main}, the two types of subtrees correspond to strategies of the two players in Schmidt's game. We need the following criterion for the existence of subtrees of type (I) in establishing Proposition \ref{P:main} below, which is at the heart of the proof of Theorem \ref{T:main}.

\begin{proposition}\label{P:tree}
Let $\cT$ be an $N$-regular rooted tree with a regular $D$-coloring, and let $\cS\subset\cT$ be a subtree. Suppose that for every subtree $\cR\subset\cT$ of type (II), $\cS\cap\cR$ is infinite. Then $\cS$ contains a subtree of type (I).
\end{proposition}

\begin{proof}
We first prove that under the assumptions of the proposition, for every $h\ge0$,
\begin{align}\label{E:tree}
&\text{there exists a subtree $\cF$ of $\cS$ such that for any $\tau\in\cF_n$ with}\notag\\
&\text{$n<h$ and any $1\le i\le D$, we have $\#\cF_{\suc}(\tau)^{(i)}=1$.}
\end{align}
If $h=0$, there is nothing to prove. Assume $h\ge1$ and \eqref{E:tree} holds if $h$ is replaced by $h-1$. Let
\begin{align*}
\cS'_1=\{\tau\in\cS_1: & \text{ the intersection of $\cS(\tau)$ with every } \\ & \text{ subtree of $\cT(\tau)$ of type (II) is infinite}\}.
\end{align*}
By the induction hypothesis, if $\tau\in\cS'_1$, then $\cS(\tau)$ has a subtree $\cF_\tau$ such that for any $\tau'\in(\cF_\tau)_n$ with $n<h-1$ and any $1\le i\le D$, we have $\#(\cF_\tau)_{\suc}(\tau')^{(i)}=1$. Thus to prove \eqref{E:tree}, it suffices to prove that $(\cS'_1)^{(i)}\ne\emptyset$ for every $1\le i\le D$. Suppose on the contrary that $(\cS'_1)^{(i_0)}=\emptyset$ for some $1\le i_0\le D$. Then for every $\tau\in\cT_1^{(i_0)}$, $\cT(\tau)$ has a subtree $\cR_\tau$ of type (II) such that $\cS(\tau)\cap\cR_\tau$ is finite whenever $\tau\in\cS$. Let $\cR\subset\cT$ be the subtree such that $\cR_1=\cT_1^{(i_0)}$ and $\cR(\tau)=\cR_\tau$ for every $\tau\in\cR_1$. Then $\cR$ is of type (II) and
$$\cS\cap\cR=\{\text{the root of $\cT$}\}\cup\bigcup_{\tau\in\cS_1^{(i_0)}}\cS(\tau)\cap\cR_\tau$$ is finite. This contradicts the assumption of the proposition.

We now prove the proposition by considering the rooted tree $\sF$ constructed as follows. For $h\ge0$, the set $\sF_h$ of vertices of level $h$ consists of the subtrees $\cF$ of $\cS$ such that $\cF_{h+1}=\emptyset$ and $\#\cF_{\suc}(\tau)^{(i)}=1$ for any $\tau\in\cF_n$ with $n<h$ and any $1\le i\le D$. Define $\cF\in\sF_{h+1}$ to be a successor of $\cF'\in\sF_h$ whenever $\cF'=\bigcup_{n=0}^h\cF_n$. In view of \eqref{E:tree}, we have $\sF_h\ne\emptyset$ for every $h\ge0$. By K\"{o}nig's lemma (see \cite[Lemma 8.1.2]{Di}), $\sF$ has an infinite path starting from the root. This means that there exists a family of subtrees $\{\cF(h)\in\sF_h: h\ge0\}$ such that $\cF(h)=\bigcup_{n=0}^h\cF(h+1)_n$ for every $h$. It follows that $\bigcup_{h=0}^{\infty}\cF(h)$ is a subtree of type (I) contained in $\cS$.
\end{proof}

\section{The winning strategy}\label{S:3}

In this section, we review the notion of a winning set for Schmidt's game, introduce a height function on rational points, and prove Theorem \ref{T:main} from Proposition \ref{P:main} below.

\subsection{Winning sets for Schmidt's game} Schmidt's game was introduced in \cite{Sc1}. It involves two real numbers $\alpha,\beta\in(0,1)$ and is played by two players, say Alice and Bob. Restricting the attention to $\RR^2$, Bob starts the game by choosing a closed disk $\cB_0\subset\RR^2$. After $\cB_n$ is chosen, Alice chooses a closed disk $\cA_n\subset \cB_n$ with $\rho(\cA_n)=\alpha\rho(\cB_n)$, and Bob chooses a closed disk $\cB_{n+1}\subset \cA_n$ with $\rho(\cB_{n+1})=\beta\rho(\cA_n)$, where $\rho(\cdot)$ denotes the radius of a disk. A subset $X\subset\RR^2$ is \emph{$(\alpha,\beta)$-winning} if Alice can play so that the single point in $\bigcap_{n=0}^\infty \cA_n=\bigcap_{n=0}^\infty \cB_n$ lies in $X$, and is \emph{$\alpha$-winning} if it is $(\alpha,\beta)$-winning for any $\beta\in(0,1)$.

\subsection{A height function on rational points}
We introduce a height function on $\QQ^2$ that play a crucial role in proving Theorem \ref{T:main}. For this, we consider rational lines in $\RR^2$ of the form
$$L(A,B,C)=\{(x,y)\in\RR^2: Ax+By+C=0\},$$
where $A,B,C\in\ZZ$ and $(A,B)\ne(0,0)$. It is natural to make the convention that when a rational line is expressed as above, then $A,B,C$ are coprime. Thus the vector $(A,B,C)$ is determined by $L(A,B,C)$ up to a negative sign. We also assume that when a point in $\QQ^2$ is expressed as $(\frac{p}{q},\frac{r}{q})$, then $q>0$ and the integers $p$, $q$, $r$ are coprime. Let $s,t\ge0$ be such that $s+t=1$. The following simple lemma is a baby version of \cite[Lemma 1]{BPV}.

\begin{lemma}\label{L:aug}
To each $P=(\frac{p}{q},\frac{r}{q})\in\QQ^2$, one can attach a rational line $L_P=L(A_P,B_P,C_P)$ passing through $P$ such that
\begin{equation}\label{E:v*}
|A_P|\le q^s, \qquad |B_P|\le q^t.
\end{equation}
\end{lemma}

We now define the height function $H:\QQ^2\to\NN$ as follows.

\begin{definition}
The \emph{height} of a rational point $P=(\frac{p}{q},\frac{r}{q})$ is
\begin{equation}\label{E:height}
H(P)=q\max\{|A_P|,|B_P|\}.
\end{equation}
\end{definition}

It follows from \eqref{E:v*} that
\begin{equation}\label{E:height-bound}
q\le H(P)\le q^{1+\max\{s,t\}}.
\end{equation}

\subsection{The winning strategy}
Let $\alpha_0=(24\sqrt{2})^{-1}$. To prove Theorem \ref{T:main}, we need to show that for any $\beta\in(0,1)$, Alice can win Schmidt's $(\alpha_0,\beta)$-game with target set $\Bad(s,t)$. In what follows, we describe a winning strategy for Alice.

In the first round of the game, for any choice of the closed disc $\cB_0$ made by Bob, Alice chooses the closed disc $\cA_0\subset\cB_0$ with $\rho(\cA_0)=\alpha_0\rho(\cB_0)$ arbitrarily. Let
$$l=2\rho(\cA_0), \qquad R=(\alpha_0\beta)^{-1},$$
and let $c>0$ be such that
\begin{equation}\label{E:c}
c<\min\left\{\frac{1}{6}lR^{-1},\frac{1}{16}R^{-12}\right\}.
\end{equation}
For $P=(\frac{p}{q},\frac{r}{q})\in\QQ^2$, we denote
\begin{equation}\label{E:Delta}
\Delta(P)=\left\{(x,y)\in\RR^2:\left|x-\frac{p}{q}\right|\le\frac{c}{q^{1+s}},\left|y-\frac{r}{q}\right|\le\frac{c}{q^{1+t}}\right\}.
\end{equation}
Then it is easy to see that
\begin{equation}\label{E:Badc2}
\RR^2\setminus\bigcup_{P\in\QQ^2}\Delta(P)\subset\Bad(s,t).
\end{equation}
We will show that for a suitable partition $\QQ^2=\bigcup_{n=1}^\infty\cP_n$, Alice has a strategy so that she can choose the closed disc $\cA_n$ in $\RR^2\setminus\bigcup_{P\in\cP_n}\Delta(P)$. This will ensure that the single point in $\bigcap_{n=0}^\infty \cA_n$ lies in the left hand side of \eqref{E:Badc2}, hence in $\Bad(s,t)$.

To define the appropriate partition, we use the height function defined above. For $n\ge1$, let
\begin{equation}\label{E:H_n}
H_n=6cl^{-1}R^n,
\end{equation}
and let
\begin{equation}\label{E:C_n}
\cP_n=\left\{P=\left(\frac{p}{q},\frac{r}{q}\right)\in\QQ^2: H_n\le H(P)< H_{n+1}\right\}.
\end{equation}
It follows from \eqref{E:c} that
\begin{equation}\label{E:H1}
H_1=6cl^{-1}R\le1.
\end{equation}
So $\QQ^2=\bigcup_{n=1}^\infty\cP_n$. Starting from this, we construct a Cantor-like set using squares. By a \emph{square} we mean a set of the form
$$\Sigma=\{(x,y)\in\RR^2:x_0\le x\le x_0+\ell(\Sigma), y_0\le y\le y_0+\ell(\Sigma)\},$$ where $\ell(\Sigma)>0$ is the side length of $\Sigma$.
Let $\Sigma_0$ be the circumscribed square of $\cA_0$. Then $\ell(\Sigma_0)=l$. We represent certain subsquares of $\Sigma_0$ as vertices of a regular rooted tree with a regular coloring. Let
\begin{equation}\label{E:m=12}
m=12,
\end{equation}
and let $\cT$ be an $m^2[R/m]^2$-regular rooted tree with a regular $[R/m]^2$-coloring, where $[\ \cdot \ ]$ denotes the integer part of a real number. We choose and fix
an injective map $\Phi$ from $\cT$ to the set of subsquares of $\Sigma_0$ satisfying the following conditions:
\begin{itemize}
  \item For any $n\ge0$ and $\tau\in\cT_n$, we have
\begin{equation}\label{E:lPhi}
\ell(\Phi(\tau))=lR^{-n}.
\end{equation}
  In particular, the root of $\cT$ is mapped to $\Sigma_0$.
  \item For $\tau,\tau'\in\cT$, if $\tau$ is a descendant of $\tau'$, then $\Phi(\tau)\subset\Phi(\tau')$.
  \item For any $n\ge1$ and $\tau\in\cT_{n-1}$, the interiors of the squares $\{\Phi(\tau'):\tau'\in\cT_{\suc}(\tau)\}$ are mutually disjoint, the union $\bigcup_{\tau'\in\cT_{\suc}(\tau)}\Phi(\tau')$ is a square of side length $m[R/m]lR^{-n}$, and for any $1\le i\le [R/m]^2$, the union $\bigcup_{\tau'\in\cT_{\suc}(\tau)^{(i)}}\Phi(\tau')$ is a square of side length $mlR^{-n}$.
\end{itemize}
It is easy to see that for any $\tau\in\cT_{n-1}$ with $n\ge1$ and any subsquare $\Sigma$ of $\Phi(\tau)$ of side length $2mlR^{-n}$, there exists $1\le i\le [R/m]^2$ such that $\bigcup_{\tau'\in\cT_{\suc}(\tau)^{(i)}}\Phi(\tau')\subset\Sigma$.

The Cantor-like set is constructed from the subtree $\cS$ of $\cT$ defined as follows.
Let $\cS_0=\cT_0$. If $n\ge1$ and $\cS_{n-1}$ is defined, we let
\begin{equation}\label{E:S_n}
\cS_n=\{\tau\in\cT_{\suc}(\cS_{n-1}):\Phi(\tau)\cap\bigcup_{P\in\cP_n}\Delta(P)=\emptyset\}.
\end{equation}
Then $\cS=\bigcup_{n=0}^\infty\cS_n$ is a subtree of $\cT$. This gives rise to a Cantor-like set
$$\cC=\bigcap_{n=1}^\infty\bigcup_{\tau\in\cS_n}\Phi(\tau).$$
Note that by \eqref{E:S_n}, we have
\begin{equation}\label{E:tau}
\bigcup_{\tau\in\cS_n}\Phi(\tau)\subset\RR^2\setminus\bigcup_{P\in\cP_n}\Delta(P), \qquad \forall n\ge1.
\end{equation}
Thus $\cC$ is contained in the left hand side of \eqref{E:Badc2}, and hence is contained in $\Bad(s,t)$. The winning strategy for Alice will in fact enable her to choose $\cA_n$ to be the inscribed closed disc of $\Phi(\tau)$ for some $\tau\in\cS_n$. This will imply that the single point in $\bigcap_{n=0}^\infty \cA_n$ lies in $\cC$, hence in $\Bad(s,t)$. Such a winning strategy corresponds to a subtree of $\cT$ of type (I) contained in $\cS$, whose existence is ensured by the following proposition.

\begin{proposition}\label{P:main}
The tree $\cS$ contains a subtree of type (I).
\end{proposition}

Proposition \ref{P:main} will be proved in the next section. In the rest of this section, we assume it and prove Theorem \ref{T:main}. The proof also reflects the idea that a strategy of Bob roughly corresponds to a subtree of $\cT$ of type (II).

\begin{proof}[Proof of Theorem \ref{T:main}]
Let $\cS'$ be a subtree of $\cS$ of type (I). In view of the above analysis, it suffices to prove that for every $n\ge0$,
\begin{equation}\label{E:state}
\text{Alice can choose $\cA_{n}$ to be the inscribed closed disc of $\Phi(\tau_n)$ for some $\tau_n\in\cS'_n$.}
\end{equation}
We prove this by induction.
If $n=0$, there is nothing to prove. Assume $n\ge1$ and Alice has chosen $\cA_{n-1}$ as the inscribed closed disc of $\Phi(\tau_{n-1})$, where $\tau_{n-1}\in\cS'_{n-1}$. For any choice $\cB_n\subset\cA_{n-1}$ of Bob, the inscribed square of $\cB_n$ has side length
$$\sqrt{2}\rho(\cB_n)=\sqrt{2}\beta\rho(\cA_{n-1})=\frac{\sqrt{2}}{2}\beta\ell(\Phi(\tau_{n-1}))=\frac{\sqrt{2}}{2}\beta lR^{-n+1}=2mlR^{-n}.$$
So there exists $1\le i\le [R/m]^2$ such that $\bigcup_{\tau\in\cT_{\suc}(\tau_{n-1})^{(i)}}\Phi(\tau)\subset\cB_n$. Let $\tau_n$ be the unique vertex in $\cS'_{\suc}(\tau_{n-1})^{(i)}$. Then $\Phi(\tau_n)\subset\cB_n$. Note that the radius of the inscribed closed disc of $\Phi(\tau_n)$ is equal to
$$\frac{1}{2}\ell(\Phi(\tau_n))=\frac{1}{2}R^{-1}\ell(\Phi(\tau_{n-1}))=\alpha_0\beta\rho(\cA_{n-1})=\alpha_0\rho(\cB_n).$$
Thus Alice can choose $\cA_n$ to be the inscribed closed disc of $\Phi(\tau_n)$. This proves \eqref{E:state}.
\end{proof}

\section{Proof of Proposition \ref{P:main}}\label{S:4}

In this section we prove Proposition \ref{P:main}. Without loss of generality, we may assume that
\begin{equation}\label{E:st}
s\le t.
\end{equation}
In view of \eqref{E:height-bound} and \eqref{E:C_n}, for  $P=(\frac{p}{q},\frac{r}{q})\in\cP_n$ we have
\begin{equation}\label{E:C_n'}
H_n^{\frac{1}{1+t}}\le q<H_{n+1}.
\end{equation}
We further divide each $\cP_n$ into at most $n$ parts. Let
\begin{equation}\label{E:C_n0}
\cP_{n,1}=\{P\in\cP_n: H_n^{\frac{1}{1+t}}\le q<H_n^{\frac{1}{1+t}}R^{10}\},
\end{equation}
and for $k\ge2$, let
\begin{equation}\label{E:C_nl}
\cP_{n,k}=\{P\in\cP_n: H_n^{\frac{1}{1+t}}R^{2k+6}\le q<H_n^{\frac{1}{1+t}}R^{2k+8}\}.
\end{equation}
Note that if $k\ge n+1$, then by \eqref{E:H1},
$$H_n^{\frac{1}{1+t}}R^{2k+6}\ge H_n^{\frac{1}{1+t}}R^{2n+8}=H_1^{-\frac{t}{1+t}}R^{\frac{2+t}{1+t}(n-1)+9}H_{n+1}\ge H_{n+1},$$
and it follows from \eqref{E:C_n'} that $\cP_{n,k}=\emptyset$. Hence $\cP_n=\bigcup_{k=1}^n\cP_{n,k}$\footnote{In fact, it is easy to show that $\cP_{n,k}=\emptyset$ for $k\ge\frac{t}{2(1+t)}n$. But for simplicity, we prefer to use the range of $k$ as $1\le k\le n$.}. The following lemma is a key step in the proof of Proposition \ref{P:main}. Roughly speaking, it states that those points in $\cP_{n,k}$ which are ``responsible" for the construction of the Cantor-like set $\cC$ lie on a single line.

\begin{lemma}\label{L:const}
Let $n\ge1$, $1\le k\le n$, and $\tau\in\cS_{n-k}$. Then the map $P\mt L_P$ is constant on the set $$\cP_{n,k}(\tau):=\{P\in\cP_{n,k}:\Phi(\tau)\cap\Delta(P)\ne\emptyset\}.$$
\end{lemma}

\begin{proof}
Let $P_1=(\frac{p_1}{q_1},\frac{r_1}{q_1})$ and $P_2=(\frac{p_2}{q_2},\frac{r_2}{q_2})$ be distinct points in $\cP_{n,k}(\tau)$. We need to prove that $L_{P_1}=L_{P_2}$. Suppose $L_{P_i}=L(A_i,B_i,C_i)$, $i=1,2$.
Consider the three-dimensional vectors $v_i=(\frac{p_i}{q_i},\frac{r_i}{q_i},1)$ and $w_i=(A_i,B_i,C_i)$. Note that $\langle v_i,w_i\rangle=0$, where $\langle\cdot,\cdot\rangle$ denotes the standard inner product on $\RR^3$. We first verify that
\begin{equation}\label{E:inner}
|\langle v_1,w_2\rangle|\le4cq_1^{-1}R^{\lambda_k}+12cq_2^{-1}R^{k+1},
\end{equation}
where $$\lambda_k=\begin{cases}
  10, & k=1,\\
  2, & k\ge2.
\end{cases}$$
In fact, since $\Phi(\tau)\cap\Delta(P_i)\ne\emptyset$, we have
\begin{align*}
|\langle v_1,w_2\rangle|=&|\langle v_1-v_2,w_2\rangle|\\
=&\left|A_2\left(\frac{p_1}{q_1}-\frac{p_2}{q_2}\right)+B_2\left(\frac{r_1}{q_1}-\frac{r_2}{q_2}\right)\right| \\
\le& |A_2|\left(\frac{c}{q_1^{1+s}}+\frac{c}{q_2^{1+s}}+lR^{-n+k}\right)\\
 & \qquad  +|B_2|\left(\frac{c}{q_1^{1+t}}+\frac{c}{q_2^{1+t}}+lR^{-n+k}\right) && \text{(by \eqref{E:Delta} and \eqref{E:lPhi})} \\
\le& q_2^s\left(\frac{c}{q_1^{1+s}}+\frac{c}{q_2^{1+s}}\right)+q_2^t\left(\frac{c}{q_1^{1+t}}+\frac{c}{q_2^{1+t}}\right)\\
& \qquad +2\max\{|A_2|,|B_2|\}lR^{-n+k} && \text{(by \eqref{E:v*})} \\
=& cq_1^{-1}\left(\frac{q_2^s}{q_1^s}+\frac{q_1}{q_2}+\frac{q_2^t}{q_1^t}+\frac{q_1}{q_2}\right)+2q_2^{-1}H(P_2)lR^{-n+k}  && \text{(by \eqref{E:height})} \\
\le& 4cq_1^{-1}R^{\lambda_k}+12cq_2^{-1}R^{k+1}. && \text{(by \eqref{E:C_n0}, \eqref{E:C_nl} and \eqref{E:C_n})}
\end{align*}
This proves \eqref{E:inner}.

We now prove the lemma by considering two cases.

\textbf{Case 1.} Suppose $k=1$. In this case, it follows from \eqref{E:inner}, \eqref{E:C_n0} and \eqref{E:c} that
$$q_1|\langle v_1,w_2\rangle|\le4cR^{10}+12c\frac{q_1}{q_2}R^2\le16cR^{12}<1.$$
Note that $q_1|\langle v_1,w_2\rangle|$ is a nonnegative integer. Thus $q_1|\langle v_1,w_2\rangle|=0$.
This implies that $L_{P_2}$ passes through $P_1$, hence is the line passing through $P_1$ and $P_2$. Similarly, $L_{P_1}$ is the line passing through $P_1$ and $P_2$. Hence $L_{P_1}=L_{P_2}$. This proves the $k=1$ case of the lemma.

\textbf{Case 2.} Suppose $k\ge2$. It follows from \eqref{E:st} that $s\le\frac{1}{2}$. Thus, by \eqref{E:C_nl}, we have
\begin{equation}\label{E:A}
|A_i|\le q_i^s\le H_n^{\frac{s}{1+t}}R^{k+4}.
\end{equation}
On the other hand, it follows from \eqref{E:C_nl} and \eqref{E:C_n} that
\begin{equation}\label{E:AB}
\max\{|A_i|,|B_i|\}=q_i^{-1}H(P_i)\le H_n^{-\frac{1}{1+t}}R^{-2k-6}\cdot H_{n+1}=H_n^{\frac{t}{1+t}}R^{-2k-5}.
\end{equation}
Consider the cross product
\begin{equation}\label{E:cross}
(\tilde{p}_0,\tilde{r}_0,\tilde{q}_0):=w_1\times w_2.
\end{equation}
By the triple cross product expansion, we have
$$v_1\times(w_1\times w_2)=\langle v_1,w_2\rangle w_1.$$
Comparing the first two components of the vectors on both sides, we obtain
\begin{align}
\tilde{q}_0\frac{r_1}{q_1}-\tilde{r}_0&=\langle v_1,w_2\rangle A_1, \label{E:c1}\\
\tilde{q}_0\frac{p_1}{q_1}-\tilde{p}_0&=-\langle v_1,w_2\rangle B_1. \label{E:c2}
\end{align}
Note that by \eqref{E:inner} and \eqref{E:C_nl}, we have
\begin{equation}\label{E:inner1}
|\langle v_1,w_2\rangle|\le4cq_1^{-1}R^2+12cq_2^{-1}R^{k+1}\le 16cH_n^{-\frac{1}{1+t}}R^{-k-5}.
\end{equation}

We now prove that $L_{P_1}=L_{P_2}$ by contradiction. Suppose the contrary. Then $w_1\times w_2$ is a nonzero vector.
We first consider the case where $\tilde{q}_0=0$, that is, $L_{P_1}$ is parallel to $L_{P_2}$. In this case, it follows from $w_1\times w_2\ne0$ that $\max\{|\tilde{p}_0|,|\tilde{r}_0|\}\ge1$.
On the other hand,  we have
\begin{align*}
\max\{|\tilde{p}_0|,|\tilde{r}_0|\}&=|\langle v_1,w_2\rangle|\max\{|A_1|,|B_1|\} && \text{(by \eqref{E:c1} and \eqref{E:c2})}\\
&\le|\langle v_1,w_2\rangle|\max\{|A_1|,|B_1|\}^{\frac{1}{t}}\\
&\le16cH_n^{-\frac{1}{1+t}}R^{-k-5}\cdot H_n^{\frac{1}{1+t}}R^{\frac{-2k-5}{t}} && \text{(by \eqref{E:inner1} and \eqref{E:AB})}\\
&\le16c<1.  && \text{(by \eqref{E:c})}
\end{align*}
This is a contradiction.

Next, suppose that $\tilde{q}_0\ne0$, that is, $L_{P_1}$ is not parallel to $L_{P_2}$. Let $P_0=(\frac{p_0}{q_0},\frac{r_0}{q_0})$ be the intersection point of $L_{P_1}$ and $L_{P_2}$, where $q_0>0$ and the integers $p_0$, $q_0$, $r_0$ are coprime. We prove that $\Delta(P_1)\subset\Delta(P_0)$. Firstly, note that  the vector $(\tilde{p}_0,\tilde{r}_0,\tilde{q}_0)$ is a nonzero integer multiple of $(p_0,r_0,q_0)$. Thus
\begin{align}
q_0&\le|\tilde{q}_0|\notag\\
&=|A_1B_2-A_2B_1| && \text{(by \eqref{E:cross})}\notag\\
&\le|A_1B_2|+|A_2B_1|\notag\\
&\le 2H_n^{\frac{s}{1+t}}R^{k+4}\cdot H_n^{\frac{t}{1+t}}R^{-2k-5} && \text{(by \eqref{E:A} and \eqref{E:AB})}\notag\\
&=2H_n^{\frac{1}{1+t}}R^{-k-1}. \label{E:q0}
\end{align}
Suppose $(x,y)\in\Delta(P_1)$. In view of the fact that $R=(\alpha_0\beta)^{-1}>24\sqrt{2}$, it follows that
\begin{align*}
q_0^{1+s}\left|x-\frac{p_0}{q_0}\right|\le& q_0^s\left|q_0\frac{p_1}{q_1}-p_0\right|+q_0^{1+s}\left|x-\frac{p_1}{q_1}\right|\\
\le& q_0^s|\langle v_1,w_2\rangle||B_1|+q_0^{1+s}\frac{c}{q_1^{1+s}} && \text{(by \eqref{E:c2} and \eqref{E:Delta})}\\
\le& 2H_n^{\frac{s}{1+t}}R^{-s(k+1)}\cdot16cH_n^{-\frac{1}{1+t}}R^{-k-5}\cdot H_n^{\frac{t}{1+t}}R^{-2k-5}\\
&+4H_n^{\frac{1+s}{1+t}}R^{-(1+s)(k+1)}\cdot cH_n^{-\frac{1+s}{1+t}}R^{-(1+s)(2k+6)} && \text{(by \eqref{E:q0}, \eqref{E:inner1}, \eqref{E:AB} and \eqref{E:C_nl})}\\
\le& 36cR^{-2}\le c
\end{align*}
and
\begin{align*}
q_0^{1+t}\left|y-\frac{r_0}{q_0}\right|\le& q_0^t\left|q_0\frac{r_1}{q_1}-r_0\right|+q_0^{1+t}\left|y-\frac{r_1}{q_1}\right|\\
\le& q_0^t|\langle v_1,w_2\rangle||A_1|+q_0^{1+t}\frac{c}{q_1^{1+t}} && \text{(by \eqref{E:c1} and \eqref{E:Delta})}\\
\le& 2H_n^{\frac{t}{1+t}}R^{-t(k+1)}\cdot16cH_n^{-\frac{1}{1+t}}R^{-k-5}\cdot H_n^{\frac{s}{1+t}}R^{k+4}\\
&+4H_nR^{-(1+t)(k+1)}\cdot cH_n^{-1}R^{-(1+t)(2k+6)} && \text{(by \eqref{E:q0}, \eqref{E:inner1}, \eqref{E:A} and \eqref{E:C_nl})}\\
\le&36cR^{-2}\le c.  && \text{(by \eqref{E:st})}
\end{align*}
Thus $(x,y)\in\Delta(P_0)$. This proves $\Delta(P_1)\subset\Delta(P_0)$.

Let $n_0\ge1$ be the unique integer such that $P_0\in\cP_{n_0}$. We claim that
\begin{equation}\label{E:m}
n_0\ge n-k+1.
\end{equation}
In fact, if $n_0\le n-k$, then $\cS_{n_0}$ contains an ancestor $\tau'$ of $\tau$. By \eqref{E:S_n}, we have
$$\Phi(\tau)\cap\Delta(P_1)\subset\Phi(\tau')\cap\Delta(P_0)=\emptyset.$$
This contradicts $P_1\in\cP_{n,k}(\tau)$.
In view of \eqref{E:C_n'} and \eqref{E:m}, we have
$$q_0\ge H_{n_0}^{\frac{1}{1+t}}\ge H_{n-k+1}^{\frac{1}{1+t}}=H_n^{\frac{1}{1+t}}R^{^{-\frac{k-1}{1+t}}}.$$
This contradicts \eqref{E:q0}. Thus the proof of Lemma \ref{L:const} is completed.
\end{proof}

Let $w>0$. By a \emph{strip} of width $w$, we mean a subset of $\RR^2$ of the form
$$\cL=\{\bx\in\RR^2:|\bx\cdot\mathbf{u}-a|\le w/2\},$$
where $\mathbf{u}\in\RR^2$ is a unit vector, the dot denotes the standard inner product, and $a\in\RR$. Lemma \ref{L:const} implies the following statement.

\begin{corollary}\label{C:strip}
For any $n\ge1$, $1\le k\le n$ and $\tau\in\cS_{n-k}$, there exists a strip of width $\frac{2}{3}lR^{-n}$ which contains all the rectangles $\{\Delta(P):P\in\cP_{n,k}(\tau)\}$.
\end{corollary}

\begin{proof}
By Lemma \ref{L:const}, there exists $(A,B,C)\in\ZZ^3$ with $(A,B)\ne(0,0)$ such that for any $P=(\frac{p}{q},\frac{r}{q})\in\cP_{n,k}(\tau)$, we have
$$|A|\le q^s, \quad |B|\le q^t, \quad Ap+Br+Cq=0, \quad q\max\{|A|,|B|\}\ge H_n.$$
For such $P$, if $(x,y)\in\Delta(P)$, then
\begin{align*}
|Ax+By+C|&=\left|A\left(x-\frac{p}{q}\right)+B\left(y-\frac{r}{q}\right)\right|\\
&\le|A|\left|x-\frac{p}{q}\right|+|B|\left|y-\frac{r}{q}\right|\le q^s\frac{c}{q^{1+s}}+q^t\frac{c}{q^{1+t}}=\frac{2c}{q}.
\end{align*}
Thus it follows from \eqref{E:H_n} that
$$\frac{|Ax+By+C|}{\sqrt{A^2+B^2}}\le\frac{2c}{q\max\{|A|,|B|\}}\le\frac{2c}{H_n}=\frac{1}{3}lR^{-n}.$$
This implies that $\Delta(P)$ is contained in the strip
$$\left\{(x,y)\in\RR^2:\frac{|Ax+By+C|}{\sqrt{A^2+B^2}}\le\frac{1}{3}lR^{-n}\right\},$$
which has width $\frac{2}{3}lR^{-n}$.
\end{proof}

The following lemma gives an upper bound for the number of certain squares which intersect a thin strip.

\begin{lemma}\label{L:sub}
Let $\cR\subset\cT$ be a subtree of type (II), let $n\ge1$, and let $\cL$ be a strip of width $\frac{2}{3}lR^{-n}$. Then for any $1\le k\le n$ and $\tau\in\cR_{n-k}$, we have
$$\#\{\tau'\in\cR(\tau)_k:\Phi(\tau')\cap\cL\ne\emptyset\}\le(3m-2)^k.$$
\end{lemma}

\begin{proof}
For $\cV\subset\cR$, we denote $\cV^\cL=\{\tau'\in\cV:\Phi(\tau')\cap\cL\ne\emptyset\}$. We prove the lemma by showing that
\begin{equation}\label{E:sub}
\#\cR(\tau)_{k'}^\cL\le(3m-2)^{k'}, \qquad \forall k'\in\{0,\ldots,k\}.
\end{equation}
Firstly, we note that if $0\le n'\le n-1$ and $\tau'\in\cR_{n'}$, then
\begin{equation}\label{E:sub0}
\#\cR_{\suc}(\tau')^\cL\le 3m-2.
\end{equation}
In fact, since the $m^2$ squares $\{\Phi(\tau''):\tau''\in\cR_{\suc}(\tau')\}$ have side lengths $lR^{-n'-1}$, and their union is a square of side length $mlR^{-n'-1}$, it is easy to see that a strip of width less than $\frac{\sqrt{2}}{2}lR^{-n'-1}$ intersects at most $3m-2$ squares $\Phi(\tau'')$. We now prove \eqref{E:sub} by induction on $k'$.
If $k'=0$, there is nothing to prove. Suppose that $1\le k'\le k$ and \eqref{E:sub} holds if $k'$ is replaced by $k'-1$. In view of
$$\cR(\tau)_{k'}^\cL
=\bigcup_{\tau'\in\cR(\tau)_{k'-1}^\cL}\cR_{\suc}(\tau')^\cL,$$
it follows from \eqref{E:sub0} and the induction hypothesis that $$\#\cR(\tau)_{k'}^\cL=\sum_{\tau'\in\cR(\tau)_{k'-1}^\cL}\#\cR_{\suc}(\tau')^\cL
\le (3m-2)\#\cR(\tau)_{k'-1}^\cL\le(3m-2)^{k'}.$$
This proves \eqref{E:sub}.
\end{proof}

Combining Corollary \ref{C:strip} and Lemma \ref{L:sub}, we obtain

\begin{corollary}\label{C:<=2}
Let $\cR\subset\cT$ be a subtree of type (II). Then for any $n\ge1$, $1\le k\le n$ and $\tau\in\cS_{n-k}\cap\cR_{n-k}$, we have
$$
\#\{\tau'\in\cR(\tau)_k:\Phi(\tau')\cap\bigcup_{P\in\cP_{n,k}(\tau)}\Delta(P)\ne\emptyset\}\le(3m-2)^k.
$$
\end{corollary}

We now prove Proposition \ref{P:main} using Proposition \ref{P:tree} and Corollary \ref{C:<=2}.

\begin{proof}[Proof of Proposition \ref{P:main}]
In view of Proposition \ref{P:tree}, it suffices to prove that the intersection of $\cS$ with every subtree of type (II) is infinite. Let $\cR\subset\cT$ be a subtree of type (II), and denote $a_n=\#\cS_n\cap\cR_n$. Then $a_0=1$. We prove the infinity of $\cS\cap\cR$ by showing that for any $n\ge1$,
\begin{equation}\label{E:infinity}
a_n>88a_{n-1}.
\end{equation}

It is easy to see from \eqref{E:S_n} that $\cR_{\suc}(\cS_{n-1}\cap\cR_{n-1})$ is the disjoint union of $\cS_n\cap\cR_n$ and
$$\cU_n:=\{\tau\in\cR_{\suc}(\cS_{n-1}\cap\cR_{n-1}):\Phi(\tau)\cap\bigcup_{P\in\cP_n}\Delta(P)\ne\emptyset\}.$$
Thus
\begin{equation}\label{E:infinity1}
a_n=\#\cR_{\suc}(\cS_{n-1}\cap\cR_{n-1})-\#\cU_n=m^2a_{n-1}-\#\cU_n.
\end{equation}
But
\begin{align*}
\cU_n=&\bigcup_{k=1}^n\{\tau'\in\cR_{\suc}(\cS_{n-1}\cap\cR_{n-1}):\Phi(\tau')\cap\bigcup_{P\in\cP_{n,k}}\Delta(P)\ne\emptyset\}\\
\subset&\bigcup_{k=1}^n\bigcup_{\tau\in\cS_{n-k}\cap\cR_{n-k}}\{\tau'\in\cR(\tau)_k:\Phi(\tau')\cap\bigcup_{P\in\cP_{n,k}(\tau)}\Delta(P)\ne\emptyset\}.
\end{align*}
Thus it follows from Corollary \ref{C:<=2} that
\begin{equation}\label{E:infinity2}
\#\cU_n\le\sum_{k=1}^n(3m-2)^ka_{n-k}.
\end{equation}
From \eqref{E:infinity1}, \eqref{E:infinity2} and \eqref{E:m=12}, we obtain
\begin{equation}\label{E:infinity3}
a_n\ge m^2a_{n-1}-\sum_{k=1}^n(3m-2)^ka_{n-k}=144a_{n-1}-\sum_{k=1}^n34^ka_{n-k}.
\end{equation}
By letting $n=1$ in \eqref{E:infinity3}, we see that $a_1\ge110$. So \eqref{E:infinity} holds for $n=1$. Assume $n\ge2$ and \eqref{E:infinity} holds if $n$ is replaced by $1,\ldots,n-1$.
Then for any $1\le k\le n$, we have $$a_{n-k}\le88^{-k+1}a_{n-1}.$$ Substituting this into \eqref{E:infinity3}, we obtain
$$a_n\ge \left(144-88\sum_{k=1}^n(34/88)^k\right)a_{n-1}>88a_{n-1}.$$
This proves \eqref{E:infinity}.
\end{proof}

\section*{Acknowledgments} The author would like to thank Dmitry Kleinbock for helpful comments on an early version of this paper. He is also grateful to Dzmitry Badziahin, Nikolay Moshchevitin, Andrew Pollington, Sanju Velani and Barak Weiss for valuable conversations.

\end{document}